\documentclass[a4paper,11pt]{amsart}
\usepackage[T1]{fontenc}
 \makeatletter \theoremstyle{plain}
\newtheorem{thm}{Theorem}[section]
\numberwithin{equation}{section}
\numberwithin{figure}{section} 
\theoremstyle{plain}
\newtheorem{cor}[thm]{Corollary} 
\theoremstyle{plain}
\newtheorem{lemma}[thm]{Lemma} 
\theoremstyle{plain}
\newtheorem{tnm}[thm]{Definition} 
\theoremstyle{plain}

\begin{document}

\title{Some New Paranormed  Sequence Spaces and $\alpha-$ Core}

\author{Serkan DEMIRIZ and Celal \c CAKAN}
\address{(S. Demiriz): Gazi Osmanpaþa University Faculty of Arts and Science Department of Mathematics, Tokat, TURKEY \\
(C. \c Cakan): \.In\"on\"u University Faculty of Education,
44280-Malatya, TURKEY}\email{serkandemiriz@gmail.com,
ccakan@inonu.edu.tr}

\begin{abstract}
In this study, we define new paranormed sequence spaces by combining
a double sequential band matrix and a diagonal matrix. Furthermore,
we compute the $\alpha-,\beta-$ and $\gamma-$ duals and obtain bases
for these sequence spaces. Besides this, we characterize the matrix
transformations from the new paranormed sequence spaces to the
spaces $c_{0}(q),c(q),\ell(q)$ and $\ell_{\infty}(q)$. Finally,
$\alpha-core$  of a complex-valued sequence has been introduced, and
we prove some inclusion theorems related to this new type of core.
\end{abstract} \maketitle

\section{Introduction}

By $\omega$, we shall denote the space of all real valued sequences.
Any vector subspace of $\omega$ is called as a \textit{sequence
space}. We shall write $\ell_{\infty},c$ and $c_{0}$ for the spaces
of all bounded, convergent and null sequences, respectively. Also by
$bs,cs,\ell_{1}$ and $\ell_{p}$ ; we denote the spaces of all
bounded, convergent, absolutely and $p-$ absolutely convergent
series, respectively; $1< p< \infty$.

A linear topological space $X$ over the real field $\mathbb{R}$ is
said to be a paranormed space if there is a subadditive function
$g:X\rightarrow \mathbb{R}$ such that $g(\theta)=0, g(x)=g(-x)$ and
scalar multiplication is continuous,i.e.,
$|\alpha_{n}-\alpha|\rightarrow 0$ and $g(x_{n}-x)\rightarrow 0$
imply $g(\alpha_{n}x_{n}-\alpha x)\rightarrow 0$ for all $\alpha'$s
in $\mathbb{R}$ and all $x$'s in $X$, where $\theta$ is the zero
vector in the linear space $X$.

Assume here and after that  $(p_{k})$ be a bounded sequences of
strictly positive real numbers with $\sup p_{k}=H$ and $M=\max
\{1,H\}$. Then, the linear spaces $ c(p),c_{0}(p),\ell_{\infty}(p)$
and $\ell(p)$ were defined by Maddox \cite{m1,m2} (see also Simons
\cite{simons} and Nakano \cite{nakano}) as follows:

\begin{eqnarray*}
c(p)&=&\left\{x=(x_{k})\in\omega:\lim_{k\rightarrow \infty}
|x_{k}-l|^{p_{k}}=0 ~  \textrm{for some}~ l\in\mathbb{C}\right\},\\
c_{0}(p)&=&\left\{x=(x_{k})\in\omega:\lim_{k\rightarrow \infty}
|x_{k}|^{p_{k}}=0 \right\},\\
\ell_{\infty}(p)&=&\left\{x=(x_{k})\in\omega: \sup_{k\in
\mathbb{N}}|x_{k}|^{p_{k}}<\infty\right\}
\end{eqnarray*}
and
$$
\ell(p)=\bigg\{x=(x_{k})\in \omega: \sum_{k}
|x_{k}|^{p_{k}}<\infty\bigg\},
$$
which are the complete spaces paranormed by
$$
h_{1}(x)=\sup_{k\in\mathbb{N}} |x_{k}|^{p_{k}/M} \ {\rm iff} \
\inf_{p_{k}}>0 \qquad {\rm and} \qquad
h_{2}(x)=\bigg(\sum_{k}|x_{k}|^{p_{k}}\bigg)^{1/M},
$$
respectively. We shall assume throughout that
$p_{k}^{-1}+(p_{k}^{'})^{-1}=1$ provided $1<\inf p_{k}<H<\infty$.
For simplicity in notation, here and in what follows, the summation
without limits runs from $0$ to $\infty$. By $\mathcal{F}$ and
$\mathbb{N}_{k}$, we shall denote the collection of all finite
subsets of $\mathbb{N}$ and the set of all $n\in\mathbb{N}$ such
that $n\geq k$, respectively.

For the sequence spaces $X$ and $Y$, define the set $S(X,Y)$ by
\begin{equation}\label{1.1}
S(X,Y)=\{z=(z_{k}): xz=(x_{k}z_{k})\in Y \   \ {\rm for\  \ all}\ \
x\in X\}.
\end{equation}
With the notation of (\ref{1.1}), the $\alpha-,\beta-$ and $\gamma-$
duals of a sequence space $X$, which are respectively denoted by
$X^{\alpha}, X^{\beta}$ and $X^{\gamma}$, are defined by
$$
X^{\alpha}=S(X,\ell_{1}) , \     \ X^{\beta}=S(X,cs) \ {\rm and} \ \
X^{\gamma}=S(X,bs).
$$

Let $(X,h)$ be a paranormed space. A sequence $(b_{k})$ of the
elements of $X$ is called a basis for $X$ if and only if, for each
$x\in X$, there exists a unique sequence $(\alpha_{k})$ of scalars
such that
$$
h\left (x-\sum_{k=0}^{n} \alpha_{k}b_{k}\right) \rightarrow 0 \ \ as
\     \   n\rightarrow\infty.
$$
The series $\sum \alpha_{k} b_{k}$ which has the sum $x$ is then
called the expansion of $x$ with respect to $(b_{n})$ and written as
$x=\sum \alpha_{k} b_{k}$.

Let $X,Y$ be any two sequence spaces and $A=(a_{nk})$ be an infinite
matrix of real numbers $a_{nk}$,where $n,k\in \mathbb{N}$. Then, we
say that $A$ defines a matrix mapping from $X$ into $Y$, and we
denote it by writing $A:X\rightarrow Y$, if for every sequence
$x=(x_{k})\in X$ the sequence $Ax=((Ax)_{n})$, the $A$-transform of
$x$, is in $Y$, where
\begin{equation}\label{1.2}
(Ax)_{n}=\sum_{k} a_{nk}x_{k},\     \ (n\in\mathbb{N}).
\end{equation}
By $(X:Y)$, we denote the class of all matrices $A$ such that
$A:X\rightarrow Y$. Thus, $A\in(X:Y)$ if and only if the series on
the right-hand side of (1.2) converges for each $n\in \mathbb{N}$
and every $x\in X$, and we have $Ax=\{(Ax)_{n}\}_{n\in
\mathbb{N}}\in Y$ for all $x\in X$. A sequence $x$ is said to be
$A$- summable to $\alpha$ if $Ax$ converges to $\alpha$ which is
called as the $A$- limit of $x$. If $X$ and $Y$ are equipped with
the limits $X-\lim$ and $Y-\lim$, respectively, $A\in (X: Y)$ and
$Y-\lim_n A_n(x) = X-\lim_kx_k$ for all $x\in X$, then we say that
$A$ regularly maps $X$ into $Y$ and write $A\in (X: Y)_{reg}$.

Let $x=(x_k)$ be a sequence in $\mathbb{C}$, the set of all complex
numbers, and $R_k$ be the least convex closed region of complex
plane containing $x_k, x_{k+1},x_{k+2},\ldots$. The Knopp Core (or
$\mathcal{K}-core$) of $x$ is defined by the intersection of all
$R_k$ ($k$=1,2,\ldots), (see \cite{r1}, pp.137). In \cite{rs}, it is
shown that
\begin{eqnarray*}
\mathcal{K}-core(x)= \bigcap_{z\in \mathbb{C}}B_x(z)
\end{eqnarray*}
for any bounded sequence $x$, where $B_x(z) = \big\{w\in \mathbb{C}:
|w-z|\leq \limsup_k|x_k-z|\big\}$.

Let $E$ be a subset of $\mathbb{N}$. The natural density $\delta$ of
$E$ is defined by
\begin{eqnarray*}
\delta(E) = \lim_n \frac{1}{n} |\{k \leq n: k \in E\}|
\end{eqnarray*}
where $|\{k \leq n: k \in E\}|$ denotes the number of elements of
$E$ not exceeding $n$. A sequence $x=(x_k)$ is said to be
statistically convergent to a number $l$, if $\delta(\{k: |x_k -
l|\geq \varepsilon\}) = 0$ for every $\varepsilon$. In this case we
write $st-\lim x = l$, \cite{r4}. By $st$ we denote the space of all
statistically convergent sequences.

In \cite{rfo}, the notion of the statistical core (or $st-core$) of
a complex valued sequence has been introduced by Fridy and Orhan and
it is shown for a statistically bounded sequence $x$ that
\begin{eqnarray*}
st-core(x) = \bigcap_{z\in \mathbb{C}}C_x(z),
\end{eqnarray*}
where $C_x(z) = \big\{w\in \mathbb{C}: |w-z|\leq
st-\limsup_k|x_k-z|\big\}$. The core theorems have been studied by
many authors. For instance see \cite{a,cfo,cc,hcm,hc} and the
others.

We write $\mathcal{U}=\{u\in \omega: u_{k}\neq 0 \ \textrm{for all}
\ k\}$ and $\mathcal{U}^{+}=\{u\in \omega: u_{k}>0 \ \textrm{for
all} \ k\}$; if $u\in \mathcal{U}$  then we write $1/u=(1/u_{k})$
where  $k\in \mathbb{N}$. By $e$ and $e^{(n)}$ $(n=0,1,2,...)$, we
denote the sequences such that $e_{k}=1$ for $k=0,1,...$, and
$e_{n}^{(n)}=1$ and $e_{k}^{(n)}=0$ for $k\neq n$.

An infinite matrix $T=(t_{nk})$ is said to be a triangle if
$t_{nk}=0 (k>n)$ and $t_{nn}\neq 0$ for all $n$. Let us give the
definition of some triangle limitation matrices which are needed in
the text. Let $t=(t_{k})$ be a sequence of positive reals and write
$$
Q_{n}=\sum_{k=0}^{n} t_{k}, \quad (n\in \mathbb{N}).
$$
Then the Ces\`{a}ro mean of order one, Riesz mean with respect to
the sequence $t=(t_{k})$ and $A_{r}-$ mean with $0<r<1$ are
respectively defined by the matrices $C=(c_{nk})$,
$R^{t}=(r_{nk}^{t})$ and $B(r,s)=\{b_{nk}(r,s)\}$; where
$$
c_{nk}=\left\{\begin{array}{ll}
  \displaystyle \frac{1}{n+1}, & (0 \leq k \leq n),\\
  0, & (k>n), \\
\end{array}\right.
\quad r_{nk}^{t}=\left\{\begin{array}{ll}
  \displaystyle \frac{t_{k}}{Q_{n}}, & (0 \leq k \leq n),\\
  0, & (k>n), \\
\end{array}\right.
$$
and
$$
b_{nk}(r,s)=\left\{\begin{array}{ll}
  \displaystyle r, & (k=n),\\
  \displaystyle s, & (k=n-1), \\
  \displaystyle 0, & \textrm{otherwise}
\end{array}\right.
$$
for all $k,n \in \mathbb{N}$ and $r,s\in \mathbb{R}\backslash
\{0\}$. Additionally, define the summation $S=(s_{nk})$ and the
difference matrix $\Delta^{(1)}=(\delta_{nk})$ and the double
sequential band matrix
$\widetilde{B}=B(\tilde{r},\tilde{s})=\{b_{nk}(\tilde{r},\tilde{s})\}$
by
$$
s_{nk}=\left\{\begin{array}{ll}
  \displaystyle 1, & (0 \leq k \leq n),\\
  0, & (k>n), \\
\end{array}\right.
\quad \textrm {and}\quad \delta_{nk}=\left\{\begin{array}{ll}
  \displaystyle (-1)^{n-k}, & (n-1 \leq k \leq n),\\
  0, & (0\leq k<n-1\  \textrm {or}\ k>n), \\
\end{array}\right.
$$
and
$$
b_{nk}(\tilde{r},\tilde{s})=\left\{\begin{array}{ll}
  \displaystyle r_{k}, & (k=n),\\
  \displaystyle s_{k}, & (k=n-1), \\
  \displaystyle 0,& \textrm{otherwise}
\end{array}\right.
$$
for all $k,n \in \mathbb{N}$.

Defining the diagonal matrix $D=(d_{nk})$ by $d_{nn}=1/\alpha_{n}$
for $n=0,1,...$ and $\alpha=(\alpha_{n})\in \mathcal{U}^{+}$ and
putting $\widetilde{T}=DB(\tilde{r},\tilde{s})$.

The main purpose of this study is to introduce the paranormed
sequence spaces
$s_{\alpha}^{0}(\widetilde{B},p),s_{\alpha}^{(c)}(\widetilde{B},p),$
$s_{\alpha}^{(\infty)}(\widetilde{B},p)$ and
$\ell_{\alpha}(\widetilde{B},p)$ which are the set of all sequences
whose $\widetilde{T}-$transforms are in the spaces
$c_{0}(p),c(p),\ell_{\infty}(p)$ and $\ell(p)$, respectively; where
$\widetilde{T}$ denotes the matrix
$\widetilde{T}=DB(\tilde{r},\tilde{s})=\{t_{nk}(\tilde{r},\tilde{s},\alpha)\}$
defined by
$$
t_{nk}(\tilde{r},\tilde{s},\alpha)=\left\{\begin{array}{ll}
  \displaystyle r_{n}/ \alpha_{n}, & (k=n), \\
  \displaystyle s_{n-1}/ \alpha_{n}, & (k=n-1), \\
  \displaystyle 0 & \textrm{otherwise}.
  \end{array}\right.
$$
Also, we have investigated some topological structures, which have
completeness, the $\alpha-,\beta-$ and $\gamma-$ duals, and the
bases of these sequence spaces. Besides this, we characterize some
matrix mappings on these spaces. Finally,  we have defined $\alpha-$
core of a sequence and characterized some class of matrices for
which $\alpha-core (Ax)\subseteq \mathcal{K}-core (x)$ and
$\alpha-core (Ax)\subseteq st_{A}-core(x)$ for all $x\in
\ell_{\infty}$.

\section {The Paranormed Sequence Spaces $\lambda(\widetilde{B},p)$ for $\lambda\in
\{s_{\alpha}^{0},s_{\alpha}^{(c)},s_{\alpha}^{(\infty)},\ell_{\alpha}\}$}

In this section, we define the new  sequence spaces
$\lambda(\widetilde{B},p)$ for $\lambda\in
\{s_{\alpha}^{0},s_{\alpha}^{(c)},s_{\alpha}^{(\infty)},\ell_{\alpha}\}$
derived by using the double sequential band matrix and the diagonal
matrix, and prove that these sequence spaces are the complete
paranormed linear metric spaces and compute their $\alpha-,\beta-$
and $\gamma-$ duals. Moreover, we give the basis for the spaces
$\lambda(\widetilde{B},p)$ for $\lambda\in
\{s_{\alpha}^{0},s_{\alpha}^{(c)},\ell_{\alpha}\}$.

For a sequence space $X$, the matrix domain $X_{A}$ of an infinite
matrix $A$ is defined by
\begin{equation}\label{2.1}
X_{A}=\{x=(x_{k})\in\omega:\  \ Ax\in X\}.
\end{equation}

In \cite{cm}, Choudhary and Mishra have defined the sequence space
$\overline{\ell(p)}$ which consists of all sequences such that
$S$-transforms are in $\ell(p)$, where $S=(s_{nk})$ is defined by
$$
s_{nk}=\left\{\begin{array}{cc}
  \displaystyle 1, & (0 \leq k \leq n), \\
  0, & (k>n).
\end{array}\right.
$$
Ba\c{s}ar and Altay \cite{fbba1} have recently examined the space
$bs(p)$ which is formerly defined by Ba\c{s}ar in \cite{fb} as the
set of all series whose sequences of partial sums are in
$\ell_{\infty}(p)$. More recently, Altay and Ba\c{s}ar have studied
the sequence spaces $r^{t}(p),r_{\infty}^{t}(p)$ in \cite{bafb1} and
$r_{c}^{t}(p),r_{0}^{t}(p)$ in \cite{bafb2} which are derived by the
Riesz means from the sequence spaces $\ell(p),\ell_{\infty}(p),c(p)$
and $c_{0}(p)$ of Maddox, respectively. With the notation of (2.1),
the spaces
$\overline{\ell(p)},bs(p),r^{t}(p),r_{\infty}^{t}(p),r_{c}^{t}(p)$
and $r_{0}^{t}(p)$ may be redefined by

$$
\overline{\ell(p)}=[\ell(p)]_{S},\   \
bs(p)=[\ell_{\infty}(p)]_{S},\   \ r^{t}(p)=[\ell(p)]_{R^{t}},
$$

$$
r_{\infty}^{t}(p)=[\ell_{\infty}(p)]_{R^{t}},\   \
r_{c}^{t}(p)=[c(p)]_{R^{t}},\   \ r_{0}^{t}(p)=[c_{0}(p)]_{R^{t}}.\\
$$
It is well known that paranormed spaces have more general properties
than normed spaces. In the literature, the approach of constructing
a new sequence space on the paranormed space by means of the matrix
domain of a particular limitation method has recently been employed
by several authors, e.g., Ye\c{s}ilkayagil and Ba\c{s}ar
\cite{myfb1}, Nergiz and Ba\c{s}ar \cite{hnfb1,hnfb2}, Karakaya and
Polat \cite{vkhp}, Özger and Ba\c{s}ar \cite{föfb}.

The domain of the matrix $B(r,s)$ in the classical spaces
$\ell_{\infty},c_{0}$ and $c$ has recently been studied by
Kiri\c{s}çi and Ba\c{s}ar in \cite{mkfb}. The characterizations of
compact matrix operators between some of those spaces were given by
Djolovi\'{c} in \cite{dj}. Recently difference sequence spaces have
extensively been studied, for instance in
\cite{hk,rcme,mc,eekmb1,mbeek1,mbeek2}.

The sequence spaces $s_{\alpha}^{0},s_{\alpha}^{(c)},
s_{\alpha}^{(\infty)}$ and $\ell_{\alpha}$, where
$\alpha=(\alpha_{n})\in \mathcal{U}^{+}$ were introduced by de
Malafosse and Rako\v{c}evi\'{c} in \cite{bdmem1,bdmem2} as follows:

\begin{eqnarray*}
s_{\alpha}^{0}&=&\bigg\{x=(x_{k})\in \omega: \lim_{k\rightarrow
\infty} \frac{x_{k}}{\alpha_{k}}=0\bigg\}\\
s_{\alpha}^{(c)}&=& \bigg\{x=(x_{k})\in \omega: \exists l\in
\mathbb{C}\ni \lim_{k\rightarrow \infty}
\bigg(\frac{x_{k}}{\alpha_{k}}-l\bigg)=0\bigg\}\\
s_{\alpha}^{(\infty)}&=&\bigg\{x=(x_{k})\in \omega:
\sup_{k\in \mathbb{N}} \bigg|\frac{x_{k}}{\alpha_{k}}\bigg|<\infty\bigg\}\\
\ell_{\alpha}&=&\bigg\{x=(x_{k})\in \omega: \sum_{k} \bigg|\frac{x_{k}}{\alpha_{k}}\bigg|^{p}<\infty\bigg\}\\
\end{eqnarray*}
These sequence spaces have extensively been examined by B de
Malafosse in \cite{bdm1,bdm2,bdm3,bdm4}.

Following Choudhary and Mishra \cite{cm}, Ba\c{s}ar and Altay
\cite{fbba1}, Altay and Ba\c{s}ar \cite{bafb1,bafb2}, we define the
sequence spaces $\lambda(\widetilde{B},p)$ for $\lambda\in
\{s_{\alpha}^{0},s_{\alpha}^{(c)},s_{\alpha}^{(\infty)},\ell_{\alpha}\}$
by

$$
\lambda(\widetilde{B},p)=\bigg\{x=(x_{k})\in \omega:
\bigg(\frac{r_{k}x_{k}+s_{k-1}x_{k-1}}{\alpha_{k}}\bigg)\in
\mu(p)\bigg\}.
$$
for all $k\in \mathbb{N}$ and $\mu \in
\{c_{0},c,\ell_{\infty},\ell\}$. With the notation (\ref{2.1}), we
may redefine the spaces
$s_{\alpha}^{0}(\widetilde{B},p),s_{\alpha}^{(c)}(\widetilde{B},p),$
$s_{\alpha}^{(\infty)}(\widetilde{B},p)$ and
$\ell_{\alpha}(\widetilde{B},p)$ as follows:
$$
s_{\alpha}^{0}(\widetilde{B},p)=\{c_{0}(p)\}_{\widetilde{T}}, \quad
s_{\alpha}^{(c)}(\widetilde{B},p)=\{c(p)\}_{\widetilde{T}},
$$
$$
s_{\alpha}^{(\infty)}(\widetilde{B},p)=\{\ell_{\infty}(p)\}_{\widetilde{T}},
\quad \ell_{\alpha}(\widetilde{B},p)=\{\ell(p)\}_{\widetilde{T}}.
$$
In the case $p=e$, the sequence space $\lambda(\widetilde{B},p)$ is
reduced to $\lambda(\widetilde{B})$ which is introduced by E.
Malkowsky et all. \cite{emföaa} for $\lambda \in \{s_{\alpha}^{0},
s_{\alpha}^{(c)}, s_{\alpha}^{(\infty)}\}$. On the other hand, it is
clear that $\Delta$ can be obtained as a special case of
$B(\tilde{r},\tilde{s})$ for $\tilde{r}=e$ and $\tilde{s}=-e$ and it
is also trivial that $B(\tilde{r},\tilde{s})$ reduces to $B(r,s)$ in
the special case $\tilde{r}=re$ and $\tilde{s}=se$. So, the results
related to the domain of the matrix $B(\tilde{r},\tilde{s})$ are
more general and more comprehensive than the corresponding ones of
the domains of the matrices $\Delta$ and $B(r,s)$.

Define the sequence $y=(y_{n})$, which will be frequently used as
the $\widetilde{T}=DB(\tilde{r},\tilde{s})-$transform of a sequence
$x=(x_{n})$, i.e.
\begin{equation}\label{2.2}
y_{n}=\widetilde{T}_{n}x=\frac{r_{n}x_{n}+s_{n-1}x_{n-1}}{\alpha_{n}};
\quad (n\in \mathbb{N}).
\end{equation}
Since the proof may also be obtained in the similar way as for the
other spaces, to avoid the repetition of the similar statements, we
give the proof only for one of those spaces. Now, we may begin with
the following theorem which is essential in the study.

\begin{thm}\label{t2.1}
(i) The sequence spaces $\lambda(\widetilde{B},p)$ for $\lambda \in
\{s_{\alpha}^{0}, s_{\alpha}^{(c)}, s_{\alpha}^{(\infty)}\}$ are the
complete linear metric spaces paranormed by $g$, defined by
$$
g(x)=\sup_{k\in \mathbb{N}}
\bigg|\frac{r_{k}x_{k}+s_{k-1}x_{k-1}}{\alpha_{k}} \bigg|^{p_{k}/M}.
$$
$g$ is a paranorm for the spaces $s_{\alpha}^{(c)}(\widetilde{B},p)$
and $s_{\alpha}^{(\infty)}(\widetilde{B},p)$ only in the trivial
case $\inf
p_{k}>0$ .\\
(ii) $\ell_{\alpha}(\widetilde{B},p)$ is a complete linear metric
space paranormed by
$$
g^{*}(x)=\bigg(\sum_{k}
\bigg|\frac{r_{k}x_{k}+s_{k-1}x_{k-1}}{\alpha_{k}}
\bigg|^{p_{k}}\bigg)^{1/M}.
$$
\end{thm}
\begin{proof}
We prove the theorem for the space
$s_{\alpha}^{(0)}(\widetilde{B},p)$. The linearity of
$s_{\alpha}^{(0)}(\widetilde{B},p)$ with respect to the
coordinatewise addition and scalar multiplication follows from the
following inequalities which are satisfied for $x,z\in
s_{\alpha}^{(0)}(\widetilde{B},p)$ (see \cite [p.30]{m}):
\begin{eqnarray}\label{2.3}
\sup_{k\in \mathbb{N}} \bigg|
\frac{r_{k}(x_{k}+z_{k})+s_{k-1}(x_{k-1}+z_{k-1})}{\alpha_{k}}\bigg|^{p_{k}/M}
&\leq&\sup_{k\in \mathbb{N}}
\bigg|\frac{r_{k}x_{k}+s_{k-1}x_{k-1}}{\alpha_{k}}
\bigg|^{p_{k}/M}\nonumber\\
&+&\sup_{k\in \mathbb{N}}
\bigg|\frac{r_{k}z_{k}+s_{k-1}z_{k-1}}{\alpha_{k}} \bigg|^{p_{k}/M}
\end{eqnarray}
and for any $\beta \in \mathbb{R}$ (see \cite{m2}),
\begin{equation}\label{2.4}
|\beta|^{p_{k}}\leq \max \{1,|\beta|^{M}\}.
\end{equation}
It is clear that $g(\theta)=0$ and $g(x)=g(-x)$ for all $x\in
s_{\alpha}^{(0)}(\widetilde{B},p)$. Again the inequalities
(\ref{2.3}) and (\ref{2.4}) yield the subadditivity of $g$ and
$$
g(\beta x)\leq \max \{1,|\beta|\}g(x).
$$
Let $\{x^{n}\}$ be any sequence of the points $x^{n} \in
s_{\alpha}^{(0)}(\widetilde{B},p)$ such that $g(x^{n}-x)\rightarrow
0$ and $(\beta_{n})$ also be any sequence of scalars such that
$\beta_{n}\rightarrow \beta$. Then, since the inequality
$$
g(x^{n})\leq g(x)+g(x^{n}-x)
$$
holds by the subadditivity of $g$, $\{g(x^{n})\}$ is bounded and we
thus have
\begin{eqnarray*}
g(\beta_{n} x^{n}-\beta x)&=& \sup_{k\in \mathbb{N}}\left
|\frac{r_{k}(\beta_{n}x_{k}^{n}-\beta x_{k})-s_{k-1}(\beta_{n}x_{k-1}^{n}-\beta x_{k-1})}
{\alpha_{k}}\right |^{p_{k}/M}\\
&\leq& |\beta_{n}-\beta|\ \ g(x^{n})+ |\beta|\ \ g(x^{n}-x),
\end{eqnarray*}
which tends to zero as $n\rightarrow \infty$. That is to say that
the scalar multiplication is continuous. Hence, $g$ is a paranorm on
the space $s_{\alpha}^{(0)}(\widetilde{B},p)$.

It remains to prove the completeness of the space $
s_{\alpha}^{(0)}(\widetilde{B},p)$. Let $\{x^{i}\}$ be any Cauchy
sequence in the space $s_{\alpha}^{(0)}(\widetilde{B},p)$, where
$x^{i}=\{x_{0}^{(i)}, x_{1}^{(i)}, x_{2}^{(i)},...\}$. Then, for a
given $\varepsilon>0$ there exists a positive integer
$n_{0}(\varepsilon)$ such that
$$
g(x^{i}-x^{j})< \frac{\varepsilon}{2}
$$
for all $i,j\geq n_{0}(\varepsilon)$. We obtain by using definition
of $g$ for each fixed $k\in \mathbb{N}$ that
\begin{eqnarray}\label{2.5}
\big|\{\widetilde{T}x^{i}\}_{k}-\{\widetilde{T}x^{j}\}_{k}\big|^{p_{k}/M}
&\leq& \sup_{k\in \mathbb{N}}
\big|\{\widetilde{T}x^{i}\}_{k}-\{\widetilde{T}x^{j}\}_{k}\big|^{p_{k}/M}\nonumber\\
&<&\frac{\varepsilon}{2}
\end{eqnarray}
for every $i,j\geq n_{0}(\varepsilon)$, which leads us to the fact
that $\{(\widetilde{T}x^{0})_{k},(\widetilde{T}x^{1})_{k},...\}$ is
a Cauchy sequence of real numbers for every fixed $k\in \mathbb{N}$.
Since $\mathbb{R}$ is complete, it converges, say
$$
\{\widetilde{T}x^{i}\}_{k}\rightarrow \{\widetilde{T}x\}_{k}
$$
as $i\rightarrow \infty$. Using these infinitely many limits
$(\widetilde{T}x)_{0},(\widetilde{T}x)_{1},...$, we define the
sequence $\{(\widetilde{T}x)_{0},(\widetilde{T}x)_{1},...\}$. We
have from (\ref{2.5}) with $j\rightarrow \infty$ that
\begin{equation}\label{2.6}
\big|\{\widetilde{T}x^{i}\}_{k}-\{\widetilde{T}x\}_{k}\big|^{p_{k}/M}\leq
\frac{\varepsilon}{2} \quad (i\geq n_{0}(\varepsilon))
\end{equation}
for every fixed $k\in \mathbb{N}$. Since $x^{i}=\{x_{k}^{(i)}\}\in
s_{\alpha}^{(0)}(\widetilde{B},p)$,
$$
\big|\{\widetilde{T}x^{i}\}_{k}\big|^{p_{k}/M}<\frac{\varepsilon}{2}
$$
for all $k\in \mathbb{N}$. Therefore, we obtain (\ref{2.6}) that
\begin{eqnarray*}
\big|\{\widetilde{T}x\}_{k}\big|^{p_{k}/M}&\leq&
\big|\{\widetilde{T}x\}_{k}-\{\widetilde{T}x^{i}\}_{k}\big|^{p_{k}/M}
+\big|\{\widetilde{T}x^{i}\}_{k}\big|^{p_{k}/M}\\
&<&\varepsilon  \quad (i\geq n_{0}(\varepsilon)).
\end{eqnarray*}
This shows that the sequence $\{\widetilde{T}x\}$ belongs to  the
space $c_{0}(p)$. Since $\{x^{i}\}$ was an arbitrary Cauchy
sequence, the space $s_{\alpha}^{(0)}(\widetilde{B},p)$ is complete
and this concludes the proof.
\end{proof}

\begin{thm}\label{t2.2}
The sequence spaces
$s_{\alpha}^{(\infty)}(\widetilde{B},p),s_{\alpha}^{(c)}(\widetilde{B},p),s_{\alpha}^{0}(\widetilde{B},p)$
and $\ell_{\alpha}(\widetilde{B},p)$ are linearly isomorphic to the
spaces $\ell_{\infty}(p),c(p),c_{0}(p)$ and $\ell(p)$, respectively,
where $0<p_{k}\leq H<\infty$.
\end{thm}

\begin{proof}
We establish this for the space
$s_{\alpha}^{(\infty)}(\widetilde{B},p)$. To prove the theorem, we
should show the existence of a linear bijection between the spaces
$s_{\alpha}^{(\infty)}(\widetilde{B},p)$ and $\ell_{\infty}(p)$ for
$0<p_{k}\leq H<\infty$. With the notation of (\ref{2.2}), define the
transformations $T$ from $s_{\alpha}^{(\infty)}(\widetilde{B},p)$ to
$\ell_{\infty}(p)$ by $x\mapsto y=Tx$. The linearity of $T$ is
trivial. Further, it is obvious that $x=\theta$ whenever $Tx=\theta$
and hence $T$ is injective.

Let $y=(y_{k})\in \ell_{\infty}(p)$ and define the sequence
$x=(x_{k})$ by
\begin{equation}\label{2.7}
x_{k}=\sum_{j=0}^{k}
\frac{(-1)^{k-j}\alpha_{j}}{r_{k}}\prod_{i=j}^{k-1}
\frac{s_{i}}{r_{i}}; \quad (k\in \mathbb{N}).
\end{equation}
Then, we get that
\begin{eqnarray*}
g(x)&=&\sup_{k\in \mathbb{N}}
\bigg|\frac{r_{k}x_{k}+s_{k-1}x_{k-1}}{\alpha_{k}} \bigg|^{p_{k}/M}\\
&=&\sup_{k\in \mathbb{N}} \bigg|\frac{r_{k}\sum_{j=0}^{k}
\frac{(-1)^{k-j}\alpha_{j}}{r_{k}}\prod_{i=j}^{k-1}
\frac{s_{i}}{r_{i}}+s_{k-1}\sum_{j=0}^{k-1}
\frac{(-1)^{k-1-j}\alpha_{j}}{r_{k-1}}\prod_{i=j}^{k-2}
\frac{s_{i}}{r_{i}}}{\alpha_{k}} \bigg|^{p_{k}/M}\\
&=& \sup_{k\in \mathbb{N}}|y_{k}|^{p_{k}/M}=h_{1}(y)<\infty.
\end{eqnarray*}
Thus, we deduce that $x\in s_{\alpha}^{(\infty)}(\widetilde{B},p)$
and consequently $T$ is surjective and is paranorm preserving.
Hence, $T$ is a linear bijection and this says us that the spaces
$s_{\alpha}^{(\infty)}(\widetilde{B},p)$ and $\ell_{\infty}(p)$ are
linearly isomorphic, as desired.
\end{proof}

We shall quote some lemmas which are needed in proving related to
the duals our theorems.

\begin{lemma}\label{l2.3}\cite [Theorem 5.1.1 with $q_n=1$]{kgge}
$A\in (c_{0}(p):\ell(q))$  if and only if
\begin{equation}
\sup_{K \in \mathcal{F}} \sum_{n}\left|\sum_{k\in K}
a_{nk}B^{-1/p_{k}}\right|<\infty, \    \ (\exists B\in
\mathbb{N}_{2}).
\end{equation}
\end{lemma}

\begin{lemma}\label{l2.4}\cite [Theorem 5.1.9 with $q_n=1$]{kgge}
$A\in (c_{0}(p):c(q))$  if and only if
\begin{equation}
\sup_{n \in \mathcal{\mathbb{N}}} \sum_{k}
|a_{nk}|B^{-1/p_{k}}<\infty \ \ (\exists B\in \mathbb{N}_{2}),
\end{equation}
\begin{equation}
\exists (\beta_{k})\subset \mathbb{R}\ni \lim _{n\rightarrow \infty}
|a_{nk}-\beta_{k}|=0\    \ for\  \ all\  k\in \mathbb{N},
\end{equation}
\begin{equation}
\exists (\beta_{k})\subset \mathbb{R}\ni \sup _{n\in \mathbb{N}}
\sum_{k}|a_{nk}-\beta_{k}|B^{-1/p_{k}}<\infty. \quad (\exists B\in
\mathbb{N}_{2})
\end{equation}
\end{lemma}

\begin{lemma}\label{l2.5}\cite [Theorem 5.1.13 with $q_n=1$]{kgge}
$A\in (c_{0}(p):\ell_{\infty}(q))$  if and only if
\begin{equation}
\sup_{n\in \mathbb{N}}\sum_{k} |a_{nk}|B^{-1/p_{k}}<\infty. \quad
(\exists B\in \mathbb{N}_{2})
\end{equation}
\end{lemma}

\begin{lemma} \label{l2.6}\cite [Theorem 5.1.0 with $q_{n}=1$]{kgge}(i) Let $1<p_{k}\leq
H<\infty$ for all $k\in \mathbb{N}$. Then, $A\in (\ell(p):\ell_{1})$
if and only if there exists an integer $B>1$ such that
\begin{equation}\label{2.11}
\sup_{K \in \mathcal{F}} \sum_{k}\left|\sum_{n\in K}
a_{nk}B^{-1}\right|^{p_{k}^{'}}<\infty .
\end{equation}
(ii) Let $0<p_{k}\leq 1$ for all $k\in \mathbb{N}$. Then, $A\in
(\ell(p):\ell_{1})$ if and only if
\begin{equation}\label{2.12}
\sup_{K \in \mathcal{F}} \sup_{k\in \mathbb{N}}\left|\sum_{n\in K}
a_{nk}\right|^{p_{k}}<\infty .
\end{equation}
\end{lemma}

\begin{lemma} \label{l2.7}\cite [ Theorem 1 (i)-(ii)]{kgge} (i) Let $1<p_{k}\leq
H<\infty$ for all $k\in \mathbb{N}$. Then, $A\in
(\ell(p):\ell_{\infty})$ if and only if there exists an integer
$B>1$ such that

\begin{equation}\label{2.13}
\sup_{n \in \mathbb{N} } \sum_{k} |a_{nk}B^{-1}|^{p_{k}^{'}}<\infty
.
\end{equation}
(ii) Let $0<p_{k}\leq 1$ for all $k\in \mathbb{N}$. Then, $A\in
(\ell(p):\ell_{\infty})$ if and only if
\begin{equation}\label{2.14}
\sup_{n,k \in \mathbb{N}}|a_{nk}|^{p_{k}}<\infty.
\end{equation}
\end{lemma}

\begin{lemma}\label{l2.8}\cite [Corollary for Theorem 1]{kgge} Let $0<p_{k}\leq H<\infty$ for all $k\in
\mathbb{N}$. Then,  $A\in (\ell(p):c)$  if and only if (\ref{2.13}),
(\ref{2.14}) hold, and
\begin{equation}\label{2.15}
\lim_{n\rightarrow \infty} a_{nk}=\beta_{k}, \quad (k\in \mathbb{N})
\end{equation}
also holds.
\end{lemma}

\begin{thm}\label{t2.9}
Let $K^{*}=\{k\in \mathbb{N}: 0\leq k\leq n\}\cap K$ for $K\in
\mathcal{F}$ and $B\in \mathbb{N}_{2}$. Define the sets
$S_{1}^{\alpha}(r,s),S_{2}^{\alpha}(r,s),S_{3}^{\alpha}(r,s),S_{4}^{\alpha}(r,s),S_{5}^{\alpha}(r,s),S_{6}^{\alpha}(r,s)$
and $S_{7}^{\alpha}(r,s)$ as follows:
\begin{eqnarray*}
S_{1}^{\alpha}(r,s)&=&\bigcup_{B>1} \bigg\{a=(a_{k})\in \omega:
\sup_{K\in \mathcal{F}} \sum_{n} \bigg|\sum_{k\in K^{*}}
\frac{(-1)^{n-k}\alpha_{k}}{r_{n}}\prod_{j=k}^{n-1}
\frac{s_{j}}{r_{j}}a_{n}B^{-1/p_{k}}\bigg|<\infty
\bigg\}\\
S_{2}^{\alpha}(r,s)&=&\bigg\{a=(a_{k})\in \omega: \sum_{n}
\bigg|\sum_{k=0}^{n}
\frac{(-1)^{n-k}\alpha_{k}}{r_{n}}\prod_{j=k}^{n-1}
\frac{s_{j}}{r_{j}}a_{n}\bigg|<\infty\bigg\}\\
S_{3}^{\alpha}(r,s)&=&\bigcup_{B>1} \bigg\{a=(a_{k})\in \omega:
\sup_{n\in \mathbb{N}} \sum_{k=0}^{n}\bigg|\sum_{j=k}^{n}
\frac{(-1)^{j-k}\alpha_{k}}{r_{j}}\prod_{i=k}^{j-1}
\frac{s_{i}}{r_{i}}a_{j}\bigg|B^{-1/p_{k}}<\infty
\bigg\}\\
S_{4}^{\alpha}(r,s)&=&\bigg\{a=(a_{k})\in \omega:
\bigg|\sum_{j=k}^{\infty}
\frac{(-1)^{j-k}\alpha_{k}}{r_{j}}\prod_{i=k}^{j-1}
\frac{s_{i}}{r_{i}}a_{j}\bigg|<\infty \quad
\textrm{for all} \quad k\in \mathbb{N}\bigg\}\\
S_{5}^{\alpha}(r,s)&=& \bigcup_{B>1} \bigg\{a=(a_{k})\in \omega:
\exists (\beta_{k})\subset \mathbb{R}\ni \sup_{n\in \mathbb{N}}
\sum_{k=0}^{n} \bigg|\sum_{j=k}^{n}
\frac{(-1)^{j-k}\alpha_{k}}{r_{j}}\prod_{i=k}^{j-1}
\frac{s_{i}}{r_{i}}a_{j}-\beta_{k}\bigg|B^{-1/p_{k}}<\infty
\bigg\}\\
S_{6}^{\alpha}(r,s)&=&\bigg\{a=(a_{k})\in \omega: \exists \beta \in
\mathbb{R}\ni \lim_{n\rightarrow \infty} \bigg|\sum_{k=0}^{n}
\sum_{j=k}^{n} \frac{(-1)^{j-k}\alpha_{k}}{r_{j}}\prod_{i=k}^{j-1}
\frac{s_{i}}{r_{i}}a_{j}-\beta
\bigg|=0\bigg\}\\
S_{7}^{\alpha}(r,s)&=&\bigg\{a=(a_{k})\in \omega: \sup_{n\in
\mathbb{N}} \bigg|\sum_{k=0}^{n} \sum_{j=k}^{n}
\frac{(-1)^{j-k}\alpha_{k}}{r_{j}}\prod_{i=k}^{j-1}
\frac{s_{i}}{r_{i}}a_{j}\bigg|<\infty\bigg\}
\end{eqnarray*}
Then,

(i)
$\{s_{\alpha}^{0}(\widetilde{B},p)\}^{\alpha}=S_{1}^{\alpha}(r,s)$
\qquad (ii)
$\{s_{\alpha}^{(c)}(\widetilde{B},p)\}^{\alpha}=S_{1}^{\alpha}(r,s)\cap
S_{2}^{\alpha}(r,s)$ \\

(iii)
$\{s_{\alpha}^{0}(\widetilde{B},p)\}^{\beta}=S_{3}^{\alpha}(r,s)\cap
S_{4}^{\alpha}(r,s)\cap S_{5}^{\alpha}(r,s)$ \\

(iv) $\{s_{\alpha}^{(c)}(\widetilde{B},p)\}^{\beta}=
\{s_{\alpha}^{0}(\widetilde{B},p)\}^{\beta}\cap S_{6}^{\alpha}(r,s)$\\

(v)
$\{s_{\alpha}^{0}(\widetilde{B},p)\}^{\gamma}=S_{3}^{\alpha}(r,s)$
\quad (vi)
$\{s_{\alpha}^{(c)}(\widetilde{B},p)\}^{\gamma}=S_{3}^{\alpha}(r,s)\cap
S_{7}^{\alpha}(r,s)$
\end{thm}

\begin{proof}
We give the proof for the space $s_{\alpha}^{0}(\widetilde{B},p)$.
Let us take any $a=(a_{n})\in \omega$ and define the matrix
$C^{\alpha}=\{c_{nk}^{\alpha}(r,s)\}$ via the sequence $a=(a_{n})$
by
$$
c_{nk}^{\alpha}(r,s)=\left\{\begin{array}{ll}
  \displaystyle \frac{(-1)^{n-k}\alpha_{k}}{r_{n}}\prod_{j=k}^{n-1}
\frac{s_{j}}{r_{j}}a_{n}, & 0 \leq k \leq   n,\\
   \displaystyle 0, & k>n,
\end{array}\right.
$$
where $n,k\in \mathbb{N}$. Bearing in mind (\ref{2.7}) we
immediately derive that
\begin{eqnarray}\label{2.16}
a_{n}x_{n}=\sum_{k=0}^{n}
\frac{(-1)^{n-k}\alpha_{k}}{r_{n}}\prod_{j=k}^{n-1}
\frac{s_{j}}{r_{j}}a_{n}y_{k}=(C^{\alpha}y)_{n}; \quad (n\in
\mathbb{N}).
\end{eqnarray}
We therefore observe by (\ref{2.16}) that $ax=(a_{n}x_{n})\in
\ell_{1}$ whenever $x\in s_{\alpha}^{0}(\widetilde{B},p)$ if and
only if $Cy\in \ell_{1}$ whenever $y\in c_{0}(p)$. This means that
$a=(a_{n})\in \{s_{\alpha}^{0}(\widetilde{B},p)\}^{\alpha}$ whenever
$x=(x_{n})\in s_{\alpha}^{0}(\widetilde{B},p)$ if and only if
$C^{\alpha}\in (c_{0}(p):\ell_{1})$. Then, we derive by Lemma
\ref{l2.3} that
$$
\{s_{\alpha}^{0}(\widetilde{B},p)\}^{\alpha}=S_{1}^{\alpha}(r,s).
$$

Consider the equation for $n\in \mathbb{N}$,
\begin{eqnarray}\label{2.17}
\sum_{k=0}^{n}a_{k}x_{k}&=&\sum_{k=0}^{n}a_{k} \bigg[\sum_{j=0}^{k}
\frac{(-1)^{k-j}\alpha_{j}}{r_{k}}\prod_{i=j}^{k-1}
\frac{s_{i}}{r_{i}}y_{j}\bigg]\nonumber\\
&=&\sum_{k=0}^{n} \bigg[\sum_{j=k}^{n}
\frac{(-1)^{j-k}\alpha_{k}}{r_{j}}\prod_{i=k}^{j-1}
\frac{s_{i}}{r_{i}}a_{j}\bigg]y_{k} \nonumber\\
&=&(D^{\alpha}y)_{n}
\end{eqnarray}
where $D^{\alpha}=\{d_{nk}^{\alpha}(r,s)\}$ is defined by
$$
d_{nk}=\left\{\begin{array}{ll}
  \displaystyle \sum_{j=k}^{n}
\frac{(-1)^{j-k}\alpha_{k}}{r_{j}}\prod_{i=k}^{j-1}
\frac{s_{i}}{r_{i}}a_{j}, & 0 \leq k \leq n,\\
 \displaystyle 0, & k>n,
\end{array}\right.
$$
where $n,k\in \mathbb{N}$. Thus, we deduce from Lemma \ref{l2.4}
with (\ref{2.17}) that $ax=(a_{k}x_{k})\in cs$ whenever
$x=(x_{k})\in s_{\alpha}^{0}(\widetilde{B},p)$ if and only if
$D^{\alpha} y\in c$ whenever $y=(y_{k})\in c_{0}(p)$. This means
that $a=(a_{n})\in \{s_{\alpha}^{0}(\widetilde{B},p)\}^{\beta}$
whenever $x=(x_{n})\in s_{\alpha}^{0}(\widetilde{B},p)$ if and only
if $D^{\alpha}\in (c_{0}(p):c)$. Therefore we derive from Lemma
\ref{l2.4} that
$$\{s_{\alpha}^{0}(\widetilde{B},p)\}^{\beta}=S_{3}^{\alpha}(r,s)\cap
S_{4}^{\alpha}(r,s)\cap S_{5}^{\alpha}(r,s).$$

As this, we deduce from Lemma \ref{l2.5} with (\ref{2.17}) that
$ax=(a_{k}x_{k})\in bs$ whenever $x=(x_{k})\in
s_{\alpha}^{0}(\widetilde{B},p)$ if and only if $D^{\alpha} y\in
\ell_{\infty}$ whenever $y=(y_{k})\in c_{0}(p)$. This means that
$a=(a_{n})\in \{s_{\alpha}^{0}(\widetilde{B},p)\}^{\gamma}$ whenever
$x=(x_{n})\in s_{\alpha}^{0}(\widetilde{B},p)$ if and only if
$D^{\alpha}\in (c_{0}(p):\ell_{\infty})$. Therefore we obtain Lemma
\ref{l2.5} that
$$
\{s_{\alpha}^{0}(\widetilde{B},p)\}^{\gamma}=S_{3}^{\alpha}(r,s)
$$
and this completes the proof.
\end{proof}

\begin{thm}
Let $K^{*}=\{k\in \mathbb{N}: 0\leq k\leq n\}\cap K$ for $K\in
\mathcal{F}$ and $B\in \mathbb{N}_{2}$. Define the sets
$S_{8}^{\alpha}(r,s),S_{9}^{\alpha}(r,s),S_{10}^{\alpha}(r,s)$ and
$S_{11}^{\alpha}(r,s)$ as follows:
\begin{eqnarray*}
S_{8}^{\alpha}(r,s)&=&\bigcap_{B>1} \bigg\{a=(a_{k})\in \omega:
\sup_{K\in \mathcal{F}} \sum_{n} \bigg|\sum_{k\in
K^{*}}\sum_{j=k}^{n}
\frac{(-1)^{j-k}\alpha_{k}}{r_{j}}\prod_{i=k}^{j-1}
\frac{s_{i}}{r_{i}}a_{j}B^{1/p_{k}}\bigg|<\infty\bigg\}\\
S_{9}^{\alpha}(r,s)&=&\bigcap_{B>1} \bigg\{a=(a_{k})\in
\omega:\sup_{n\in \mathbb{N}} \sum_{k=0}^{n} \bigg|\sum_{j=k}^{n}
\frac{(-1)^{j-k}\alpha_{k}}{r_{j}}\prod_{i=k}^{j-1}
\frac{s_{i}}{r_{i}}a_{j}\bigg|B^{1/p_{k}}<\infty\bigg\}\\
S_{10}^{\alpha}(r,s)&=& \bigcap_{B>1} \bigg\{a=(a_{k})\in \omega:
\exists (\beta_{k})\subset \mathbb{R}\ni \lim_{n\rightarrow \infty}
\sum_{k=0}^{n} \bigg|\sum_{j=k}^{n}
\frac{(-1)^{j-k}\alpha_{k}}{r_{j}}\prod_{i=k}^{j-1}
\frac{s_{i}}{r_{i}}a_{j}-\beta_{k}\bigg|B^{1/p_{k}}=0
\bigg\}\\
S_{11}^{\alpha}(r,s)&=&\bigcap_{B>1} \bigg\{a=(a_{k})\in \omega:
\sup_{n\in \mathbb{N}}\sum_{k=0}^{n} \bigg|\sum_{j=k}^{n}
\frac{(-1)^{j-k}\alpha_{k}}{r_{j}}\prod_{i=k}^{j-1}
\frac{s_{i}}{r_{i}}a_{j}\bigg|B^{1/p_{k}}<\infty\bigg\}
\end{eqnarray*}
Then,

(i) $\{s_{\alpha}^{(\infty)}(\widetilde{B},p)\}^{\alpha}=S_{8}^{\alpha}(r,s)$\\

(ii)
$\{s_{\alpha}^{(\infty)}(\widetilde{B},p)\}^{\beta}=S_{9}^{\alpha}(r,s)\cap
S_{10}^{\alpha}(r,s)$\\

(iii)
$\{s_{\alpha}^{(\infty)}(\widetilde{B},p)\}^{\gamma}=S_{11}^{\alpha}(r,s)$.
\end{thm}

\begin{proof}
This may be obtained in the similar way, as mentioned in the proof
of Theorem \ref{t2.9} with Lemmas \ref{l2.6}(i), \ref{l2.7}(i),
\ref{l2.8} instead of Lemmas \ref{l2.3}-\ref{l2.5}. So, we omit the
details.
\end{proof}

\begin{thm}\label{t2.11}
Let $K^{*}=\{k\in \mathbb{N}: 0\leq k\leq n\}\cap K$ for $K\in
\mathcal{F}$ and $B\in \mathbb{N}_{2}$. Define the sets
$S_{12}^{\alpha}(r,s),S_{13}^{\alpha}(r,s),S_{14}^{\alpha}(r,s),S_{15}^{\alpha}(r,s)$
and $S_{16}^{\alpha}(r,s)$ as follows:
\begin{eqnarray*}
S_{12}^{\alpha}(r,s)&=&\bigg\{a=(a_{k})\in \omega: \sup_{K\in
\mathcal{F}} \sup_{k\in \mathbb{N}} \bigg|\sum_{n\in K^{*} }
\sum_{j=k}^{n} \frac{(-1)^{j-k}\alpha_{k}}{r_{j}}\prod_{i=k}^{j-1}
\frac{s_{i}}{r_{i}}a_{j}\bigg|^{p_{k}}<\infty\bigg\}\\
S_{13}^{\alpha}(r,s)&=&\bigcup_{B>1} \bigg\{a=(a_{k})\in \omega:
\sup_{K\in \mathcal{F}} \sum_{k} \bigg|\sum_{n\in K} \sum_{j=k}^{n}
\frac{(-1)^{j-k}\alpha_{k}}{r_{j}}\prod_{i=k}^{j-1}
\frac{s_{i}}{r_{i}}a_{j}B^{-1}\bigg|^{p_{k}^{'}}<\infty\bigg\}\\
S_{14}^{\alpha}(r,s)&=&\bigcup_{B>1} \bigg\{a=(a_{k})\in \omega:
\sup_{n\in \mathbb{N}} \sum_{k=0}^{n}  \bigg|\sum_{j=k}^{n}
\frac{(-1)^{j-k}\alpha_{k}}{r_{j}}\prod_{i=k}^{j-1}
\frac{s_{i}}{r_{i}}a_{j}B^{-1}\bigg|^{p_{k}^{'}}<\infty\bigg\}\\
S_{15}^{\alpha}(r,s)&=&\bigg\{a=(a_{k})\in \omega: \sup_{n,k\in
\mathbb{N}} \bigg|\sum_{j=k}^{n}
\frac{(-1)^{j-k}\alpha_{k}}{r_{j}}\prod_{i=k}^{j-1}
\frac{s_{i}}{r_{i}}a_{j}\bigg|^{p_{k}}<\infty\bigg\}\\
S_{16}^{\alpha}(r,s)&=&\bigg\{a=(a_{k})\in \omega:
\lim_{n\rightarrow \infty} \sum_{j=k}^{n}
\frac{(-1)^{j-k}\alpha_{k}}{r_{j}}\prod_{i=k}^{j-1}
\frac{s_{i}}{r_{i}}a_{j} \quad \textrm{exists}\bigg\}
\end{eqnarray*}
Then,

(i)
$$
\{\ell_{\alpha}(\widetilde{B},p)\}^{\alpha}=\left\{\begin{array}{ll}
  \displaystyle S_{12}^{\alpha}(r,s), & 0<p_{k}\leq 1
  \\
    S_{13}^{\alpha}(r,s), & 1<p_{k}\leq H<\infty
\end{array}\right.
$$

(ii)
$$
\{\ell_{\alpha}(\widetilde{B},p)\}^{\gamma}=\left\{\begin{array}{ll}
  \displaystyle S_{15}^{\alpha}(r,s), & 0<p_{k}\leq 1
  \\
    S_{14}^{\alpha}(r,s), & 1<p_{k}\leq H<\infty.
\end{array}\right.
$$

(iii) Let $0<p_{k}\leq H<\infty$. Then,
$$
\{\ell_{\alpha}(\widetilde{B},p)\}^{\beta}=S_{14}^{\alpha}(r,s)\cap
S_{15}^{\alpha}(r,s)\cap S_{16}^{\alpha}(r,s).
$$
\end{thm}

\begin{proof}
This may be obtained in the similar way, as mentioned in the proof
of Theorem \ref{t2.9} with Lemmas \ref{l2.6}(ii), \ref{l2.7}(ii),
\ref{l2.8} instead of Lemmas \ref{l2.3}-\ref{l2.5}. So, we omit the
details.
\end{proof}

Now, we may give the sequence of the points of the spaces
$s_{\alpha}^{0}(\widetilde{B},p)$, $\ell_{\alpha}(\widetilde{B},p)$
and $s_{\alpha}^{(c)}(\widetilde{B},p)$ which forms a Schauder basis
for those spaces. Because of the isomorphism $T$, defined in the
proof of Theorem \ref{t2.2}, between the sequence spaces
$s_{\alpha}^{0}(\widetilde{B},p)$ and $c_{0}(p)$,
$\ell_{\alpha}(\widetilde{B},p)$ and $\ell(p)$,
$s_{\alpha}^{(c)}(\widetilde{B},p)$ and $c(p)$ is onto, the inverse
image of the basis of the spaces $c_{0}(p),\ell(p)$ and $c(p)$ is
the basis for our new spaces $s_{\alpha}^{0}(\widetilde{B},p)$,
$\ell_{\alpha}(\widetilde{B},p)$ and
$s_{\alpha}^{(c)}(\widetilde{B},p)$, respectively. Therefore, we
have:

\begin{thm}
Let $\mu_{k}=(\widetilde{T}x)_k$ for all $k\in \mathbb{N}$. We
define the sequence $b^{(k)}=\{b_{n}^{(k)}\}_{n\in \mathbb{N}}$ for
every fixed $k\in \mathbb{N}$ by
$$
b_{n}^{(k)}=\left\{\begin{array}{ll}
  \displaystyle \frac{(-1)^{n-k}\alpha_{k}}{r_{n}}\prod_{j=k}^{n-1}\frac{s_{j}}{r_{j}}, & n\geq k,\\
    0, & n<k.
\end{array}\right.
$$
Then,\\
(a) The sequence $\{b^{(k)}\}_{k\in \mathbb{N}}$ is a basis for the
space $s_{\alpha}^{0}(\widetilde{B},p)$ and any $x\in
s_{\alpha}^{0}(\widetilde{B},p)$ has a unique representation in the
form
$$
x=\sum_{k} \mu_{k} b^{(k)}.
$$
(b) The sequence $\{b^{(k)}\}_{k\in \mathbb{N}}$ is a basis for the
space $\ell_{\alpha}(\widetilde{B},p)$ and any $x\in
\ell_{\alpha}(\widetilde{B},p)$ has a unique representation in the
form
$$
x=\sum_{k} \mu_{k} b^{(k)}.
$$
(c) The set $\{z,b^{(k)}\}$ is a basis for the space
$s_{\alpha}^{(c)}(\widetilde{B},p)$ and any $x\in
s_{\alpha}^{(c)}(\widetilde{B},p)$ has a unique representation in
the form
$$
x=lz+\sum_{k} (\mu_{k}-l)b^{(k)}
$$
where $l=\lim_{k\rightarrow \infty} (\widetilde{T}x)_k$ and
$z=(z_k)$ with
$$
z_{k}=\sum_{j=0}^{k}\frac{(-1)^{k-j}\alpha_{j}}{r_{k}}\prod_{i=j}^{k-1}\frac{s_{i}}{r_{i}}.
$$
\end{thm}

\section{Some Matrix Mappings on the Sequence Spaces $s_{\alpha}^{0}(\widetilde{B},p),s_{\alpha}^{(c)}(\widetilde{B},p)$
$,s_{\alpha}^{(\infty)}(\widetilde{B},p)$ and
$\ell_{\alpha}(\widetilde{B},p)$ }

In this section, we characterize some matrix mappings on the spaces
$s_{\alpha}^{0}(\widetilde{B},p),s_{\alpha}^{(c)}(\widetilde{B},p)$
$,s_{\alpha}^{(\infty)}(\widetilde{B},p)$ and
$\ell_{\alpha}(\widetilde{B},p)$. Firstly, we may give the following
theorem which is useful for deriving the characterization of the
certain matrix classes.

\begin{thm}\cite[Theorem 4.1]{mkfb}
Let $\lambda$ be an FK-space, $U$ be a triangle, $V$ be its inverse
and $\mu$ be arbitrary subset of $\omega$. Then we have $A\in
(\lambda_{U}:\mu)$ if and only if
\begin{equation}
E^{(n)}=(e_{mk}^{(n)})\in (\lambda:c) \quad \textrm{for all} \quad
n\in \mathbb{N}
\end{equation}
and
\begin{equation}
E=(e_{nk})\in (\lambda:\mu)
\end{equation}
where
$$
e_{mk}^{(n)}=\left\{\begin{array}{ll}
  \displaystyle \sum_{j=k}^{m} a_{nj}v_{jk}, & 0\leq k\leq m,\\
    0, & k>m,
\end{array}\right.
$$
and
$$
e_{nk}=\sum_{j=k}^{\infty} a_{nj}v_{jk} \quad \textrm{for all} \quad
k,m,n \in \mathbb{N}.
$$
\end{thm}

Now, we may quote our theorems on the characterization of some
matrix classes concerning with the sequence spaces
$s_{\alpha}^{0}(\widetilde{B},p),s_{\alpha}^{(c)}(\widetilde{B},p)$
and $s_{\alpha}^{(\infty)}(\widetilde{B},p)$. The necessary and
sufficient conditions characterizing the matrix mappings between the
sequence spaces of Maddox are determined by Grosse-Erdmann
\cite{kgge}. Let $N$ and $K$ denote the finite subset of
$\mathbb{N}$, $L$ and $M$ also denote the natural numbers. Prior to
giving the theorems, let us suppose that $(q_{n})$ is a
non-decreasing bounded sequence of positive numbers and consider the
following conditions:

\begin{equation}\label{mt23}
\lim_{m\rightarrow \infty} \sum_{j=k}^{m}
\frac{(-1)^{j-k}\alpha_{k}}{r_{j}}\prod_{i=k}^{j-1}\frac{s_{i}}{r_{i}}a_{nj}=e_{nk},
\end{equation}

\begin{equation}\label{mt24}
\forall L, \quad \sum_{k} |e_{nk}|L^{1/p_{k}}<\infty,
\end{equation}

\begin{equation}\label{mt25}
\exists (\beta_{k})\subset \mathbb{R}\ni \lim_{m\rightarrow
\infty}\bigg| \sum_{j=k}^{m}
\frac{(-1)^{j-k}\alpha_{k}}{r_{j}}\prod_{i=k}^{j-1}\frac{s_{i}}{r_{i}}a_{nj}-\beta_{k}\bigg|=0
\quad \textrm{for all}\quad  k\in \mathbb{N},
\end{equation}

\begin{equation}\label{mt26}
\exists M, \quad \sup_{m\in \mathbb{N}} \sum_{k=0}^{m}
\bigg|\sum_{j=k}^{m}
\frac{(-1)^{j-k}\alpha_{k}}{r_{j}}\prod_{i=k}^{j-1}\frac{s_{i}}{r_{i}}a_{nj}\bigg|M^{-1/p_{k}}<\infty,
\end{equation}

\begin{equation}\label{mt27}
\forall L, \exists M, \sup_{m\in \mathbb{N}} \sum_{k=0}^{m}
\bigg|\sum_{j=k}^{m}
\frac{(-1)^{j-k}\alpha_{k}}{r_{j}}\prod_{i=k}^{j-1}\frac{s_{i}}{r_{i}}a_{nj}\bigg|
L^{1/q_{n}} M^{-1/p_{k}}<\infty,
\end{equation}

\begin{equation}\label{mt28}
\lim_{m\rightarrow \infty} \sum_{k} \bigg|\sum_{j=k}^{m}
\frac{(-1)^{j-k}\alpha_{k}}{r_{j}}\prod_{i=k}^{j-1}\frac{s_{i}}{r_{i}}a_{nj}-\beta\bigg|=0,
\end{equation}

\begin{equation}\label{mt29}
\forall L, \quad \sup_{n\in \mathbb{N}} \sum_{k}
|e_{nk}|L^{1/p_{k}}<\infty,
\end{equation}

\begin{equation}\label{mt30}
\lim_{n\rightarrow \infty} e_{nk}=\beta_{k} \qquad \textrm{for all}
\quad k\in \mathbb{N},
\end{equation}

\begin{equation}\label{mt31}
\forall L, \quad \lim_{n\rightarrow \infty} \sum_{k}
|e_{nk}|L^{1/p_{k}}<\infty,
\end{equation}

\begin{equation}\label{mt32}
\forall L, \quad \lim_{n\rightarrow \infty} \sum_{k}
|e_{nk}|L^{1/p_{k}}=0,
\end{equation}

\begin{equation}\label{mt33}
\exists M, \quad \sup_{n\in \mathbb{N}} \bigg(\sum_{k\in K}
|e_{nk}|M^{-1/p_{k}}\bigg)^{q_{n}}<\infty,
\end{equation}

\begin{equation}\label{mt34}
\lim_{n\rightarrow \infty} |e_{nk}|^{q_{n}}=0, \quad \textrm{for
all} \quad k\in \mathbb{N},
\end{equation}

\begin{equation}\label{mt35}
\forall L, \exists M, \quad \sup_{n\in \mathbb{N}} \sum_{k}
|e_{nk}|L^{1/q_{n}}M^{-1/p_{k}}<\infty,
\end{equation}

\begin{equation}\label{mt36}
\lim_{n\rightarrow \infty} |e_{nk}-\beta_{k}|^{q_{n}}=0, \quad
\textrm{for all} \quad k\in \mathbb{N},
\end{equation}

\begin{equation}\label{mt37}
\exists M, \quad \sup_{n\in \mathbb{N}} \sum_{k}
|e_{nk}|M^{-1/p_{k}}<\infty,
\end{equation}

\begin{equation}\label{mt38}
\forall L, \exists M, \quad \sup_{n\in \mathbb{N}} \sum_{k}
|e_{nk}-\beta_{k}|L^{1/q_{n}}M^{-1/p_{k}}<\infty,
\end{equation}

\begin{equation}\label{mt39}
\sup_{n\in \mathbb{N}} \bigg|\sum_{k} e_{nk}\bigg|^{q_{n}}<\infty,
\end{equation}

\begin{equation}\label{mt40}
\lim_{n\rightarrow \infty} \bigg|\sum_{k} e_{nk}\bigg|^{q_{n}}=0,
\end{equation}

\begin{equation}\label{mt41}
\lim_{n\rightarrow \infty} \bigg|\sum_{k}
e_{nk}-\beta\bigg|^{q_{n}}=0,
\end{equation}

\begin{thm}

(i) $A\in (s_{\alpha}^{(\infty)}(\widetilde{B},p):\ell_{\infty})$ if
and only
if (\ref{mt23}),(\ref{mt24}) and (\ref{mt29}) hold.\\

(ii) $A\in (s_{\alpha}^{(\infty)}(\widetilde{B},p):c)$ if and only
if
(\ref{mt23}),(\ref{mt24}), (\ref{mt30}) and (\ref{mt31}) hold.\\

(iii) $A\in (s_{\alpha}^{(\infty)}(\widetilde{B},p):c_{0})$ if and
only if (\ref{mt23}),(\ref{mt24}) and (\ref{mt32}) hold.
\end{thm}

\begin{thm}

(i) $A\in (s_{\alpha}^{0}(\widetilde{B},p):\ell_{\infty}(q))$ if and
only if
(\ref{mt25}), (\ref{mt26}), (\ref{mt27}) and (\ref{mt33}) hold.\\

(ii)  $A\in (s_{\alpha}^{0}(\widetilde{B},p):c_{0}(q))$ if and only
if (\ref{mt25}), (\ref{mt26}), (\ref{mt27}), (\ref{mt34}) and
(\ref{mt35}) hold.\\

(iii) $A\in (s_{\alpha}^{0}(\widetilde{B},p):c(q))$ if and only if
(\ref{mt25}), (\ref{mt26}), (\ref{mt27}), (\ref{mt36}), (\ref{mt37})
and (\ref{mt38}) hold.

\end{thm}

\begin{thm}
(i)  $A\in (s_{\alpha}^{(c)}(\widetilde{B},p):\ell_{\infty}(q))$ if
and only if (\ref{mt25}), (\ref{mt26}), (\ref{mt27}), (\ref{mt28}),
(\ref{mt33}) and (\ref{mt39}) hold.

(ii) $A\in (s_{\alpha}^{(c)}(\widetilde{B},p):c_{0}(q))$ if and only
if (\ref{mt25}), (\ref{mt26}), (\ref{mt27}), (\ref{mt28}),
(\ref{mt34}), (\ref{mt35})
and (\ref{mt40}) hold.\\

(iii) $A\in (s_{\alpha}^{(c)}(\widetilde{B},p):c(q))$ if and only if
(\ref{mt25}), (\ref{mt26}), (\ref{mt27}), (\ref{mt28}),
(\ref{mt36}), (\ref{mt37}), (\ref{mt38}) and (\ref{mt41}) hold.
\end{thm}

\section{$\alpha-$Core }

Using the convergence domain of the matrix
$\widetilde{T}=\{t_{nk}^{\alpha}(r,s)\}$, the new sequence spaces
$s_{\alpha}^{0}(\widetilde{B})$ and
$s_{\alpha}^{(c)}(\widetilde{B})$ have been constructed  and their
some properties have been investigated in \cite{emföaa}. In this
section we will consider the sequences with complex entries and by
$\ell_{\infty}(\mathbb{C})$ denote the space of all bounded complex
valued sequences.

Following Knopp, a core theorem is characterized a class of matrices
for which the core of the transformed sequence is included by the
core of the original sequence. For example Knopp Core Theorem
\cite[p. 138]{r1} states that $\mathcal{K}-core(Ax)\subseteq
\mathcal{K}-core(x)$ for all real valued sequences $x$ whenever $A$
is a positive matrix in the class $(c:c)_{reg}$.

Here, we will define $\alpha-core$ of a complex valued sequence and
characterize the class of matrices to yield
$\alpha-core(Ax)\subseteq \mathcal{K}-core(x)$ and
$\alpha-core(Ax)\subseteq st-core(x)$ for all
$x\in\ell_{\infty}(\mathbb{C})$.

Now, let us write
$$
\tau_{n}(x)=\frac{r_{n}x_{n}+s_{n-1}x_{n-1}}{\alpha_{n}}
$$
where $n\in \mathbb{N}$. Then, we can define $\alpha-core$ of a
complex sequence as follows:

\begin{tnm}
Let $H_{n}$ be the least closed convex hull containing
$\tau_{n}(x),\tau_{n+1}(x),\tau_{n+2}(x),...$. Then, $\alpha-core$
of $x$ is the intersection of all $H_{n}$, i.e.,
$$
\alpha-core(x)=\bigcap _{n=1}^{\infty} H_{n}.
$$
\end{tnm}
Note that, actually, we define $\alpha-core$ of $x$ by the
$\mathcal{K}-core$ of the sequence $(\tau_{n}(x))$. Hence, we can
construct the following theorem which is an analogue of
$\mathcal{K}-core$, \cite{rs}:

\begin{thm}
For any $z\in \mathbb{C}$, let
$$
G_{x}(z)= \left\{\omega\in \mathbb{C}: |\omega-z|\leq \limsup_{n}
|\tau_{n}(x)-z|\right\}.
$$
Then, for any $x\in \ell_{\infty}$,
$$
\alpha-core(x)=\bigcap _{z\in \mathbb{C}}G_{x}(z).
$$
\end{thm}
Now, we prove some lemmas which will be useful to the main results
of this section. To do these, we need to characterize the classes
$(c:s_{\alpha}^{(c)}(\widetilde{B}))_{reg}$ and $(st(A)\cap
\ell_{\infty}: s_{\alpha}^{(c)}(\widetilde{B}))_{reg}$. For brevity,
in what follows we write $\widetilde{b}_{nk}$ in place of

$$
\frac{r_{n}b_{nk}+s_{n-1}b_{n-1,k}}{\alpha_{n}}; \quad (n\geq 1).
$$

\begin{lemma}\label{l4.1}
$B\in (\ell_{\infty}:s_{\alpha}^{(c)}(\widetilde{B}))$ if and only
if
\begin{equation}\label{4.1}
\|B\|=\sup_{n} \sum_{k}|\tilde{b}_{nk}|<\infty,
\end{equation}
\begin{equation}\label{4.2}
\lim_{n} \tilde{b}_{nk}=\beta_{k} \ {\rm for \ each} \ k,
\end{equation}
\begin{equation}\label{4.3}
\lim_{n} \sum_{k} |\tilde{b}_{nk}-\beta_{k}|=0.
\end{equation}
\end{lemma}

\begin{lemma}\label{l4.2}
$B\in (c:s_{\alpha}^{(c)}(\widetilde{B}))_{reg}$ if and only if
(\ref{4.1}) and (\ref{4.2}) of the Lemma \ref{l4.1} hold with
$\beta_{k}=0$ for all $k\in \mathbb{N}$ and
\begin{equation}\label{4.5}
\lim_{n} \sum_{k}\tilde{b}_{nk}=1.
\end{equation}
\end{lemma}
\begin{lemma}\label{l4.3}
$B\in (st(A)\cap
\ell_{\infty}:s_{\alpha}^{(c)}(\widetilde{B}))_{reg}$ if and only if
$B\in (c:s_{\alpha}^{(c)}(\widetilde{B}))_{reg}$ and
\begin{equation}\label{4.6}
\lim_{n} \sum_{k\in E} |\tilde{b}_{nk}|=0
\end{equation}
for every $E\subset \mathbb{N}$ with $\delta_{A}(E)=0$.
\end{lemma}
\begin{proof}
Because of $c\subset st\cap \ell_{\infty}$, $B\in (c :
s_{\alpha}^{(c)}(\widetilde{B}))_{reg}$. Now, for any $x\in
\ell_{\infty}$ and a set $E\subset \mathbb{N}$ with $\delta(E)=0$,
let us define the sequence $z=(z_{k})$ by
$$
z_{k}=\left\{\begin{array}{cc}
  \displaystyle x_{k}, & k\in E \\
  0, & k \notin E.
\end{array}\right.
$$
Then, since $z\in st_{0}$, $Az\in s_{\alpha}^{(0)}(\widetilde{B})$,
where $s_{\alpha}^{(0)}(\widetilde{B})$ is the space of sequences
which the $\widetilde{T}-$ transforms of them in $c_{0}$. Also,
since
$$
\sum_{k}\tilde{b}_{nk}z_{k}= \sum_{k\in E}\tilde{b}_{nk}x_{k},
$$
the matrix $D=(d_{nk})$ defined by $d_{nk}=\tilde{b}_{nk} \quad
(k\in E)$ and $d_{nk}=0 \quad (k\notin E)$ is in the class
$(\ell_{\infty} : s_{\alpha}^{(c)}(\widetilde{B}))$. Hence, the
necessity of (\ref{4.6}) follows from Lemma \ref{l4.1}.

Conversely, let $x\in st(A)\cap \ell_{\infty}$ with $st_{A}-\lim
x=l$. Then, the set $E$ defined by $E=\{k: |x_{k}-l|\geq
\varepsilon\}$ has density zero and $|x_{k}-l|\leq \varepsilon$ if
$k\notin E$. Now, we can write

\begin{equation}\label{4.7}
\sum_{k}\tilde{b}_{nk}x_{k}=\sum_{k}\tilde{b}_{nk}(x_{k}-l)+l\sum_{k}\tilde{b}_{nk}.
\end{equation}
Since
$$
\left|\sum_{k}\tilde{b}_{nk}(x_{k}-l)\right|\leq \|x\|\sum_{k\in
E}|\tilde{b}_{nk}|+\varepsilon\cdot \|B\|,
$$
letting $n\rightarrow \infty$ in (\ref{4.7}) and using (\ref{4.5})
with (\ref{4.6}), we have
$$
\lim_{n} \sum_{k} \tilde{b}_{nk}x_{k}=l.
$$
This implies that $B\in (st(A)\cap \ell_{\infty} :
s_{\alpha}^{(c)}(\widetilde{B}))_{reg}$ and the proof is completed.
\end{proof}

Now, we may give some inclusion theorems. Firstly, we need a lemma.

\begin{lemma}\cite[Corollary 12]{ss}\label{l4.6}
Let $A=(a_{nk})$ be a matrix satisfying $\sum_{k} |a_{nk}|<\infty$
and $\lim_{n} a_{nk}=0$. Then, there exists an $y\in \ell_{\infty}$
with $\|y\|\leq 1$ such that
$$
\limsup_{n} \sum_{k} a_{nk}y_{k}= \limsup_{n} \sum_{k} |a_{nk}|.
$$
\end{lemma}
\begin{thm}\label{t4.7}
Let $B\in (c:s_{\alpha}^{(c)}(\widetilde{B}))_{reg}$. Then, $\alpha-
core(Bx)\subseteq \mathcal{K}-core(x)$  for all $x\in \ell_{\infty}$
if and only if
\begin{equation}\label{4.8}
\lim_{n} \sum_{k} |\tilde{b}_{nk}|=1.
\end{equation}
\end{thm}
\begin{proof}
Since $B\in (c:s_{\alpha}^{(c)}(\widetilde{B}))_{reg}$, the matrix
$\widehat{B}=(\tilde{b}_{nk})$ is satisfy the conditions of Lemma
\ref{l4.6}. So, there exists a $y\in \ell_{\infty}$ with $\|y\|\leq
1$ such that

$$
\left \{\omega\in \mathbb{C}: |\omega|\leq
\limsup_{n}\sum_{k}\tilde{b}_{nk} y_{k} \right\}= \left \{\omega\in
\mathbb{C}: |\omega|\leq \limsup_{n}\sum_{k}|\tilde{b}_{nk}|
\right\}.
$$
On the other hand, since $\mathcal{K}-core(y)\subseteq B_{1}(0)$, by
the hypothesis
$$
\left\{\omega\in \mathbb{C}: |\omega|\leq
\limsup_{n}\sum_{k}|\tilde{b}_{nk}| \right\}\subseteq B_{1}(0)=\left
\{\omega\in \mathbb{C}: |\omega|\leq 1\right\}
$$
which implies (\ref{4.8}).\\

Conversely, let $\omega\in \alpha- core(Bx)$. Then, for any given
$z\in \mathbb{C}$, we can write
\begin{eqnarray}
|\omega -z|&\leq& \limsup_{n} |\tau_{n} (Bx)-z| \label{4.9}\\
&=&\limsup_{n} \left|z-\sum_{k}\tilde{b}_{nk}x_{k}\right|\nonumber\\
&\leq& \limsup_{n} \left|\sum_{k}\tilde{b}_{nk}(z-x_{k})\right|+
\limsup_{n}|z|\left|1-\sum_{k}\tilde{b}_{nk} \right| \nonumber\\
&=&\limsup_{n}
\left|\sum_{k}\tilde{b}_{nk}(z-x_{k})\right|.\nonumber
\end{eqnarray}
Now, let $\limsup_{k} |x_{k}-z|=l$. Then, for any $\varepsilon >0$,
$|x_{k}-z|\leq l+\varepsilon$ whenever $k\geq k_{0}$. Hence, one can
write that

\begin{eqnarray}
\bigg|\sum_{k}\tilde{b}_{nk}(z-x_{k})\bigg|&=&\left|\sum_{k<k_{0}}\tilde{b}_{nk}(z-x_{k})+\sum_{k\geq
k_{0}}\tilde{b}_{nk}(z-x_{k})\right| \label{4.10}\\
&\leq& \sup_{k} |z-x_{k}|
\sum_{k<k_{0}}|\tilde{b}_{nk}|+(l+\varepsilon) \sum_{k\geq
k_{0}}|\tilde{b}_{nk}|\nonumber\\
&\leq& \sup_{k} |z-x_{k}|
\sum_{k<k_{0}}|\tilde{b}_{nk}|+(l+\varepsilon)
\sum_{k}|\tilde{b}_{nk}|.\nonumber
\end{eqnarray}

 Therefore, applying $\limsup_{n}$ under
the light of the hypothesis and combining (\ref{4.9}) with
(\ref{4.10}), we have
$$
|\omega -z|\leq \limsup_{n}
\left|\sum_{k}\tilde{b}_{nk}(z-x_{k})\right|\leq l+\varepsilon
$$
which means that $\omega \in \mathcal{K}-core(x)$. This completes
the proof.
\end{proof}

\begin{thm}\label{t4.8}
Let $B\in (st(A)\cap \ell_{\infty} :
s_{\alpha}^{(c)}(\widetilde{B}))_{reg}$. Then, $\alpha-
core(Bx)\subseteq st_{A}-core(x)$ for all $x\in \ell_{\infty}$ if
and only if (\ref{4.8}) holds.
\end{thm}
\begin{proof}
Since $st_{A}-core(x)\subseteq \mathcal{K}-core(x)$ for any sequence
$x$ \cite{kd}, the necessity of the condition (\ref{4.8}) follows
from Theorem \ref{t4.7}.

Conversely, take $\omega \in \alpha-core(Bx)$. Then, we can write
again (\ref{4.9}). Now; if $st_{A}-\limsup |x_{k}-z|=s$, then for
any $\varepsilon>0$, the set $E$ defined by $E=\{k: |x_{k}-z|>
s+\varepsilon\}$ has density zero, (see \cite{kd}). Now, we can
write
\begin{eqnarray}
\bigg|\sum_{k}\tilde{b}_{nk}(z-x_{k})\bigg|&=&\left|\sum_{k\in
E}\tilde{b}_{nk}(z-x_{k})+\sum_{k\notin
E}\tilde{b}_{nk}(z-x_{k})\right|\nonumber\\
&\leq& \sup_{k} |z-x_{k}| \sum_{k\in
E}|\tilde{b}_{nk}|+(s+\varepsilon) \sum_{k\notin
E}|\tilde{b}_{nk}|\nonumber\\
&\leq& \sup_{k} |z-x_{k}| \sum_{k\in
E}|\tilde{b}_{nk}|+(s+\varepsilon)
\sum_{k}|\tilde{b}_{nk}|.\nonumber
\end{eqnarray}
Thus, applying the operator $\limsup_{n}$ and using the condition
(\ref{4.8}) with (\ref{4.6})  , we get that
\begin{equation}\label{4.11}
\limsup_{n}\left |\sum_{k}\tilde{b}_{nk}(z-x_{k})\right|\leq
s+\varepsilon.
\end{equation}
Finally, combining (\ref{4.9}) with (\ref{4.11}), we have
$$
|\omega-z|\leq st_{A}-\limsup_{k} |x_{k}-z|
$$
which means that $\omega \in st_{A}-core(x) $ and the proof is
completed.
\end{proof}

As a consequence of Theorem \ref{t4.8}, we have
\begin{cor}
Let $B\in (st\cap
\ell_{\infty}:s_{\alpha}^{(c)}(\widetilde{B}))_{reg}$. Then,
$\alpha- core(Bx)\subseteq st-core(x)$ for all $x\in \ell_{\infty}$
if and only if (\ref{4.8}) holds.
\end{cor}

\vskip 1truecm
\end{document}